\documentclass[preprint,12pt,a4paper,compress]{elsarticle}
\usepackage{amsmath} 
\allowdisplaybreaks[3]
\usepackage{geometry} 
\usepackage{graphicx} 
\usepackage{lineno}  
\modulolinenumbers[8] 
\usepackage{multirow} 
\usepackage{amsthm} 
\usepackage{enumerate} 
\usepackage{enumitem} 
\usepackage[linesnumbered,lined,boxed]{algorithm2e} 
\usepackage[figuresright]{rotating}
\usepackage{threeparttable}
\usepackage{booktabs}
\usepackage{float} 
\usepackage{bm} 
\usepackage{setspace} 

\usepackage{indentfirst} 

\usepackage{amsfonts}
\usepackage{amssymb}
\usepackage{color}
\usepackage{graphicx}
\usepackage{hyperref}
\usepackage{appendix}

\newtheorem{theorem}{Theorem}[section] 
 
\newtheorem{lemma}[theorem]{Lemma} 
\newtheorem{corollary}[theorem]{Corollary}

\newtheorem{example}{Example}[section]
\newtheorem{proposition}[theorem]{Proposition}

\newtheorem{remark}[theorem]{Remark}

\numberwithin{equation}{section}

\newtheorem*{convention}{Convention}

\journal{Communications in algebra}

\makeatletter
\def\ps@pprintTitle{%
	\let\@oddhead\@empty
	\let\@evenhead\@empty
	\def\@oddfoot{\reset@font\hfil\thepage\hfil} 
	\let\@evenfoot\@oddfoot
}
\makeatother

\begin{document}
    
\begin{frontmatter}
	
	\title{Sub-structure in module category of Virasoro vertex operator algebras $L(c_{p,q}, 0)$}

	\author[]{Qirui Fang}
\author[]{Yu Teng}
\author[]{Yukun Xiao\corref{mycorrespondingauthor}}
\cortext[mycorrespondingauthor]{Corresponding author}
\ead{ykxiao@qdu.edu.cn}
\author[]{Wen Zheng}
\address{School of Mathematics and Statistics, Qingdao University, Qingdao 266071, China\\
 Qingdao International Mathematics Center, Qingdao, 266071, China}

	\begin{abstract}
		We determine all premodular subcategories and modular tensor subcategories in the module categories of Virasoro vertex operator algebras $L(c_{p,q},0)$ and the module categories of the simple current extensions of $L(c_{p,p+1},0)$.
	\end{abstract}
	
	
\end{frontmatter}

\section{Introduction}
The birth of vertex operator algebra theory and modular tensor category theory was
partially motivated by two-dimensional conformal field theory (see \cite{FLM,MM}). Among numerous connections between them, the module category of a rational  $C_2$-cofinite vertex operator algebra is a modular tensor category over $\mathbb{C}$ (see \cite{H1,H2}). 

A modular tensor category $\mathcal{C}$ is said to be prime if every semisimple modular tensor subcategory is equivalent to either $\mathcal{C}$  or the trivial modular category $\operatorname{Vec}_{\mathbb{C}}$. There is the prime factorization introduced by Müger in \cite{M} which states that every modular tensor category is equivalent to the Deligne products of prime modular tensor categories. In this sense, modular tensor subcategories are extremely important since they could be regarded as the skeleton of the original modular tensor category. Motivated by these aspects, we determine all modular tensor subcategories in the module categories of Virasoro vertex operator algebras $L(c_{p,q},0)$ which are denoted by $\mathcal{C}_{p,q}$ and the module categories of the simple current extensions of $L(c_{p,p+1},0)$ which are denoted by $\widetilde{\mathcal{C}}_p$.
\begin{convention}\emph{
		In this article, the subcategory $\mathcal{D}$ of modular tensor category $\mathcal{C}$ is a  nontrivial full subcategory whose class of simple objects is a subcollection of the simple objects in $\mathcal{C}$.    } 
\end{convention}
Apparently, not every subcategory of modular tensor category is still modular. Consider the simplest example $\mathcal{C}_{3,4}$ which represents for the module category of $L(\frac{1}{2},0)$. It has only three simple objects $L(\frac{1}{2},0)$, $L(\frac{1}{2},\frac{1}{2})$ and $L(\frac{1}{2},\frac{1}{16})$. By fusion rules (see \cite{DMZ, W}), the subcategory generated by $L(\frac{1}{2},0)$ and $L(\frac{1}{2},\frac{1}{2})$ is the only possible  premodular subcategory of $\mathcal{C}_{3,4}$ but it is not modular. As for general categories $\mathcal{C}_{p,q}$  and $\widetilde{\mathcal{C}}_p$, we also begin with determining all possible premodular subcategories by fusion rules and then pick the modular ones by computing the determinant of $S$-matrices (we also provide a more categorical method in Section 3 by determining the Müger center). The premodular subcategories yielded in this process are useful. As in Section 5, we glue the simple objects of these premodular subcategories (which may not be modular) to realize a large class of vertex operator subalgebras appearing in the GKO construction.

The rest of this paper is organized as follows: In Section 2, we prove that there are only 6 premodular subcategory in $\mathcal{C}_{p,q}$ when $p\geq 5$ (other cases see Remark \ref{other case premodular}). Furthermore, there are only 2 modular tensor subcategories in $\mathcal{C}_{p,q}$ when $p, q$ have different parity and 6  when $p, q$ are both odd.
In Section 3, we use a different and more categorical method to solve this problem. In particular, we have provided a criterion to determine which object lies in the Müger Center (see Lemma \ref{lm1.1}). 
Section 4 is devoting to discuss the  simple current extensions of $L(c_{p,p+1},0)$, where $p\equiv1, 2\pmod4$.  We show that there are only 2 premodular and modular tensor subcategories in $\widetilde{\mathcal{C}}_p$ for $p \neq 9,10$, and 6 for $p=9,10$. In Section 5, as an application of our results, a large class of vertex operator algebras which appear in the Deligne product of two suitable premodular subcategories determined in Section 2 are obtained.

We assume that the reader is familiar with the basic knowledge on vertex operator algebras (cf. \cite{FLM},\cite{LL}) and modular tensor categories (cf. \cite{EGNO}).

\section{Premodular and modular tensor subcategories in $\mathcal{C}_{p,q}$}

A {\it premodular category} is a $\mathbb{C}$-linear finite semisimple monoidal rigid abelian category with a braiding and a ribbon structure. A {\it modular tensor category} is a premodular category which $S$-matrix is nondegenerate. The finite order of $T$-matrix in a premodular category $\mathcal{C}$ is usually called its \textbf{Frobenius-Schur exponent}, denoted by $\text{FSexp}(\mathcal{C})$.

Denote by $\mathcal{C}_{p,q}$ the module category of Virasoro vertex operator algebra $L(c_{p,q},0)$ and by $(m,n)$ the module $L(c_{p,q},h_{m,n})$ of $L(c_{p,q},0)$, where $c_{p,q}=1-\frac{6(p-q)^2}{p q}$, $h_{m,n}=\frac{(pm-nq)^2-(p-q)^2}{4pq}$, $p,q\geq2, \text{gcd}(p,q)=1, 0<m<p$ and $0<n<q$. Without loss of generality, we assume that $p<q$ in this article. In fact, $\mathcal{C}_{p,q}$ are modular tensor categories.

\begin{proposition}\label{prop1}
	There are only six nontrivial premodular subcategories of $\mathcal{C}_{p,q}$ for any $p\geq5$.
\end{proposition}
\begin{proof}
	We use the fusion rule in \cite{W}:
	\begin{align*}
		(m,n)\boxtimes(m',n')=\bigoplus_{(m'',n'')}(m'',n''),
	\end{align*}
	where $((m'',n''),(m,n),(m',n'))$ is an admissible triple of pairs. By the definition of premodular subcategories, we only need to employ fusion rules to study the closeness of fusion tensor product.
	
	For simplicity, we first identify all the subcategories in the horizontal direction or in the vertical direction, and then use them to generate larger subcategories. A \textbf{horizontal} or \textbf{vertical} subcategory of $\mathcal{C}_{p,q}$ consists of modules of the form $(1,n)$ or $(m,n)$, respectively.
	
	\textbf{Case 1}: When $p$ and $q$ are odd, we denote by $\mathcal{D}$ a possible horizontal subcategory of $\mathcal{C}_{p,q}$. By the definition, the unit object $(1,1)$ is contained in $\mathcal{D}$.
	
	If the module $(1,2)\in\mathcal{D}$, it follows that $(1,3)\in\mathcal{D}$ because $(1,2)\boxtimes(1,2)=(1,1)\oplus(1,3)$. Repeating this process, we find that for each $n\geq4$, $(1,n)$ is a summand of $(1,2)\boxtimes(1,n-1)$ and thus $(1,n)\in \mathcal{D}$. Then we obtain the subcategory 
	\begin{align*}
		\bm{\mathcal{C}_1}=\{(1,1),(1,2),(1,3),\dots,(1,q-1)\}.
	\end{align*}
	
	If the module $(1,2)\notin\mathcal{D}$ and $(1,3)\in\mathcal{D}$, it follows that $(1,5)\in\mathcal{D}$ because $(1,3)\boxtimes(1,3)=(1,1)\oplus(1,3)\oplus(1,5)$. Repeating this process, we also find that all the $(1,2n-1)\in \mathcal{D}$, since $(1,2n-1)$ is a summand of $(1,3)\boxtimes(1,2n-3)$ for $2\leq n\leq \frac{q-1}{2}$. Then we obtain the subcategory 
	\begin{align*}\
		\bm{\mathcal{C}_2}=\{(1,1),(1,3),(1,5),\dots,(1,q-2)\}.
	\end{align*}
	
	If the module $(1,n)\in\mathcal{D}$ where $n$ is the minimal (except $n=1$) and $4\leq n<q-1$, it follows that $(1,3)$ is a summand of $(1,n)\boxtimes(1,n)$, which contradicts the minimality of $n$.
	
	Because $(1,q-1)\boxtimes(1,q-1)=(1,1)$, $\mathcal{C}_5=\{(1,1),(1,q-1)\}$ is also a subcategory of $\mathcal{C}_{p,q}$.
	
	The discussion for the vertical subcategories is similar to the horizontal ones. There are also three vertical subcategories
	\begin{align*}
		\begin{split}
			\bm{\mathcal{C}_3} =& \{(1,1),(2,1),(3,1),\dots,(p-1,1)\}, \\
			\bm{\mathcal{C}_4} =& \{(1,1),(3,1),(5,1),\dots,(p-2,1)\}, \\
			\bm{\mathcal{C}_5} =& \{(1,1),(p-1,1)\}=\{(1,1),(1,q-1)\},
		\end{split}
	\end{align*}
	because modules $(m,n)$ and $(p-m,q-n)$ are identical in \cite{W}. 
	
	Next, we use the subcategories $\mathcal{C}_1,\mathcal{C}_2,\dots,\mathcal{C}_5$ to generate the other premodular subcategories of $\mathcal{C}_{p,q}$. Since $(1,n)$(or $(1,q-n)$) and $(m,1)$(or $(p-m,1)$) are contained in $(m,n)\boxtimes(m,n)$ for $m,n>1$, we do not need to consider the other single-row or single-column. Let $(\mathcal{D},\mathcal{D}')$ denote the subcategory generated by $\mathcal{D}$ and $\mathcal{D}'$. It is obvious that
	\begin{align*}
		(m,1)\boxtimes(1,n)=(m,n).
	\end{align*}
	Thus, we can generate the remaining subcategories using one horizontal subcategory and one vertical subcategory.
	
	\begin{table}[H]
		\centering
		\begin{tabular}{|c||c c c|}
			\hline
			& $\mathcal{C}_1$ & $\mathcal{C}_2$ & $\mathcal{C}_5$ \\
			\hline\hline
			$\mathcal{C}_3$ & $\mathcal{C}_{p,q}$ & $\mathcal{C}_{p,q}$ & $\mathcal{C}_{3}$ \\
			$\mathcal{C}_4$ & $\mathcal{C}_{p,q}$ & $\mathcal{C}_{6}$ & $\mathcal{C}_{3}$ \\
			$\mathcal{C}_5$ & $\mathcal{C}_{1}$ & $\mathcal{C}_{1}$ & $\mathcal{C}_{5}$ \\
			\hline
		\end{tabular}
		\caption{Generated subcategories where $p$ and $q$ are odd}\label{tb1}
	\end{table}
	
	The results are shown above, and we have obtained a new nontrivial subcategory:
	\begin{equation*}
		\bm{\mathcal{C}_6}:=(\mathcal{C}_2,\mathcal{C}_4)=\left\{
		\begin{array}{cccc}
			(1,1) & (1,3) & \dots & (1,q-2) \\
			(3,1) & (3,3) & \dots & (3,q-2) \\
			& \dots & \dots &  \\
			(p-2,1) & (p-2,3) & \dots & (p-2,q-2) 
		\end{array}
		\right\}.
	\end{equation*}

	\textbf{Case 2}: When $p$ is odd and $q$ is even, the subcategories $\mathcal{C}_1,\mathcal{C}_3,\mathcal{C}_4$ and $\mathcal{C}_5$ are the same as in the case where $p$ and $q$ are both odd. And we have the following results:
	\begin{align*}
		\mathcal{C}_2 & =\{(1,1),(1,3),(1,5),\dots,(1,q-1)\},\\
		\mathcal{C}_6 & =\left\{
		\begin{array}{cccc}
			(1,1) & (1,3) & \dots & (1,q-1) \\
			(2,1) & (2,3) & \dots & (2,q-1) \\
			& \dots & \dots &  \\
			(\frac{p-1}{2},1) & (\frac{p-1}{2},2) & \dots & (\frac{p-1}{2},q-1)
		\end{array}
		\right\}.
	\end{align*}
	
	\begin{table}[H]
		\centering
		\begin{tabular}{|c||c c c|}
			\hline
			& $\mathcal{C}_1$ & $\mathcal{C}_2$ & $\mathcal{C}_5$ \\
			\hline\hline
			$\mathcal{C}_3$ & $\mathcal{C}_{p,q}$ & $\mathcal{C}_{6}$ & $\mathcal{C}_{3}$ \\
			$\mathcal{C}_4$ & $\mathcal{C}_{p,q}$ & $\mathcal{C}_{6}$ & $\mathcal{C}_{3}$ \\
			$\mathcal{C}_5$ & $\mathcal{C}_{1}$ & $\mathcal{C}_{2}$ & $\mathcal{C}_{5}$ \\
			\hline
		\end{tabular}
		\caption{Generated subcategories where $p$ is odd and $q$ is even}\label{tb2}
	\end{table}
	
	\textbf{Case 3}: When $p$ is even and $q$ is odd, the subcategories $\mathcal{C}_1, \mathcal{C}_2, \mathcal{C}_3$ and $\mathcal{C}_5$ are the same as in the case where $p$ and $q$ are both odd. And we also have the following results:
	\begin{align*}
		\mathcal{C}_4 & =\{(1,1),(3,1),(5,1),\dots,(p-1,1)\},\\
		\mathcal{C}_6 & =\left\{
		\begin{array}{cccc}
			(1,1) & (1,2) & \dots & (1,\frac{q-1}{2}) \\
			(3,1) & (3,3) & \dots & (3,\frac{q-1}{2}) \\
			& \dots & \dots &  \\
			(p-1,1) & (p-1,2) & \dots & (p-1,\frac{q-1}{2})
		\end{array}
		\right\}.
	\end{align*}
	
	\begin{table}[H]
		\centering
		\begin{tabular}{|c||c c c|}
			\hline
			& $\mathcal{C}_1$ & $\mathcal{C}_2$ & $\mathcal{C}_5$ \\
			\hline\hline
			$\mathcal{C}_3$ & $\mathcal{C}_{p,q}$ & $\mathcal{C}_{p,q}$ & $\mathcal{C}_{3}$ \\
			$\mathcal{C}_4$ & $\mathcal{C}_{6}$ & $\mathcal{C}_{6}$ & $\mathcal{C}_{4}$ \\
			$\mathcal{C}_5$ & $\mathcal{C}_{1}$ & $\mathcal{C}_{1}$ & $\mathcal{C}_{5}$ \\
			\hline
		\end{tabular}
		\caption{Generated subcategories where $p$ is odd and $q$ is even}\label{tb2}
	\end{table}
	
	In conclusion, $\mathcal{C}_{p,q}$ exactly has six nontrivial premodular subcategories for any $p\geq5$.
	
\end{proof}

\begin{remark}\label{other case premodular}
	\em{When $p=2$ and $q\leq5$ ,there is no any nontrivial subcategory. When $p=2$ and $q\geq7$, there is only one nontrivial subcategory: $\{(1,1),(1,\frac{q-1}{2})\}$.
	When $p=3$ and $q=4$, there is only one nontrivial subcategory: $\{(1,1),(1,3)\}$. When $p=3$ and $q\geq5$, there are only two nontrivial subcategories: $\mathcal{C}_2$ and $\mathcal{C}_5$. When $p=4$, there are only four nontrivial premodular subcategories: $\mathcal{C}_1, \mathcal{C}_2, \mathcal{C}_3$ and $\mathcal{C}_5$.}
\end{remark}

\begin{proposition}\label{prop2}
	When $p$ and $q$ have different parity, the subcategories $\mathcal{C}_5$ and $\mathcal{C}_6$ are not modular; when both $p$ and $q$ are odd, they are modular.
\end{proposition}
\begin{proof}
	Denote the $S$-matrix of $\mathcal{C}_i$ by $S^{(i)}$ (i=1,\dots,6). We need to verify the non-degeneracy of $S^{(5)}$ and $S^{(6)}$. It was proved in \cite{DLN} that the categorical $S$-matrix of category $\mathcal{C}_{p,q}$ and the genus one $S$-matrix (see \cite{Zhu})   of $L(c_{p,q},0)$ are identical up to a constant. So using the expression for the genus one $S$-matrix in calculations does not affect its non-degeneracy.
	
From the formula $(4.5)$ in \cite{DJX}, we have
	\begin{equation*}\label{eq6}
		S^{(5)}=\frac{8}{pq}\sin^2(\frac{q\pi}{p})\sin^2(\frac{p\pi}{q})\left\{
		\begin{array}{cc}
			-1 & (-1)^{p+q} \\
			(-1)^{p+q} & (-1)^{pq+1} 
		\end{array}
		\right\}.
	\end{equation*}
	Thus when $p$ and $q$ are odd, $\mathcal{C}_5$ is modular, otherwise it is not modular.
	
	When $p$ is odd and $q$ is even, the objects in $\mathcal{C}_6$ are of the form $(j,2i-1)$ with $1\leq i\leq\frac{q}{2}$ and $1\leq j\leq\frac{p-1}{2}$.
	
	\begin{align*}
		{S^{(6)}}_{(1,1)}^{(j,2i-1)} = \sqrt{\frac{8}{pq}}(-1)^j \sin(\frac{jq\pi}{p})\sin(\frac{(2i-1)p\pi}{q}),
	\end{align*}
	
	\begin{align*}
		\begin{split}
			{S^{(6)}}_{(1,q-1)}^{(j,2i-1)} = & \sqrt{\frac{8}{pq}}(-1)^{2i-1+j(q-1)+1}\sin(\frac{jq\pi}{p})\sin(\frac{p(q-1)(2i-1)\pi}{q})\\
			= & \sqrt{\frac{8}{pq}}(-1)^j \sin(\frac{jq\pi}{p})\sin(\frac{(2i-1)p\pi}{q}).
		\end{split}
	\end{align*}
	Since ${S^{(6)}}_{(1,1)}^{(j,2i-1)}={S^{(6)}}_{(1,p)}^{(j,2i-1)}$ for any $i$ and $j$, the matrix $S^{(6)}$ is degenerate.
	
	When $p$ is even and $q$ is odd, the objects in $\mathcal{C}_6$ are of the form $(2i-1,j)$ with $1\leq i\leq\frac{p}{2}$ and $1\leq j\leq\frac{q-1}{2}$.
	
	
	Following a similar calculation to that above, the equality ${S^{(6)}}_{(1,1)}^{(2i-1,j)}={S^{(6)}}_{(1,p)}^{(2i-1,j)}$ holds for all $i$ and $j$, which entails that the matrix $S^{(6)}$ is also degenerate.
	
	When $p$ and $q$ are odd, the objects in $\mathcal{C}_6$ are of the form $(2i-1,2j-1)$ with $1\leq i\leq\frac{p-1}{2}$ and $1\leq j\leq\frac{q-1}{2}$. We have the entries
	\begin{equation*}
		{S^{(6)}}_{(2i-1,2j-1)}^{(2i'-1,2j'-1)}=-\sqrt{\frac{8}{pq}}\sin(\frac{(2i-1)(2i'-1)q\pi}{p})\sin(\frac{(2j-1)(2j'-1)p\pi}{q}),
	\end{equation*}
	and consider matrix ${S^{(6)}}^{\mathbf{T}}S^{(6)}$. For $(i,j)\neq(i',j')$,
	\begin{align*}
		({S^{(6)}}^{\mathbf{T}}S^{(6)})_{(2i-1,2j-1)}^{(2i'-1,2j'-1)} & =\sum_{k,l}S_{(2i-1,2j-1)}^{(2k-1,2l-1)}S_{(2i'-1,2j'-1)}^{(2k-1,2l-1)} \\
		& = \frac{2}{pq}\left[\frac{\sin((p-1)\theta_{i-i'})}{2\sin(\theta_{i-i'})}-\frac{\sin((p-1)\theta_{i+i'-1})}{2\sin(\theta_{i+i'-1})}\right]\sum_{l}\sin(\eta_{j,l})\sin(\eta_{j',l})\\
		& = \frac{2}{pq}(-\frac{1}{2}-(-\frac{1}{2}))\sum_{l} \sin(\eta_{j,l})\sin(\eta_{j',l})\\
		& = 0,
	\end{align*}
	where $\theta_x=\frac{2xq\pi}{p}$ and $\eta_{y,z}=\frac{(2y-1)(2z-1)p\pi}{q}$. And it follows that
	\begin{align*}
		({S^{(6)}}^{\mathbf{T}}S^{(6)})_{(2i-1,2j-1)}^{(2i-1,2j-1)} & = \frac{8}{pq}\sum_{k,l}\sin^2(\frac{(2i-1)(2k-1)q\pi}{p})\sin^2(\frac{(2j-1)(2l-1)p\pi}{q})\\
		& = \frac{8}{pq}\frac{p}{4}\sum_{l}\sin^2(\frac{(2j-1)(2l-1)p\pi}{q})\\
		& = \frac{1}{2}.
	\end{align*}
	Thus, the subcategory $\mathcal{C}_6$ is modular when $p$ and $q$ are odd. 
\end{proof}

\begin{theorem}\label{thm3}
	In category $\mathcal{C}_{p,q}$, when $p$ is even and $q$ is odd with $p\geq4$, among all premodular subcategories (as specified in Proposition \ref{prop1}), only $\mathcal{C}_2$ and $\mathcal{C}_3$ are modular tensor categories. On the other hand, when $p$ is odd and $q$ is even with $p\geq5$, only $\mathcal{C}_1$ and $\mathcal{C}_4$ admit modular tensor category structures.
\end{theorem}

\begin{proof}
	When $p$ is even and $q$ is odd with $p\geq4$, since the subcategories inherit the premodular tensor category structure from $\mathcal{C}_{p,q}$, it suffices to prove that the $S$-matrices $S^{(2)}$ and $S^{(3)}$ are non-degenerate, whereas the other ones are degenerate.
	
	For $\mathcal{C}_2 = \{(1,1), (1,3),\ldots, (1,q-2)\}$, the $(i,j)\text{-}$entry of the matrix $S^{(2)}$ is given by the formula $(4.5)$ in \cite{DJX} as follows:
	\begin{equation*}
		S^{(2)}_{i,j} = -\sqrt{\frac{8}{pq}}\sin(\frac{q\pi}{p})\sin(\frac{(2i-1)(2j-1)p\pi }{q}).
	\end{equation*}
	Let $M^{(2)}_{i,j}=\sin(\frac{(2i-1)(2j-1)p\pi}{q})$. It suffices to prove that $M^{(2)}$ is non-degenerate. Consider matrix $ N^{(2)}:=(M^{(2)})^{\mathbf{T}}M^{(2)} $. For $j\neq k$,
	\begin{equation*}\label{eq12}
		\begin{aligned}
			N_{j,k}^{(2)}&=\sum_{i=1}^{\frac{q-1}{2}} \sin(\frac{(2i-1)(2j-1)p\pi}{q})\sin(\frac{(2i-1)(2k-1)p\pi }{q})\\ &=\frac{1}{2}\left[\sum_{i=1}^{l}\cos(\frac{(2i-1)(2j-2k)p\pi}{q})-\sum_{i=1}^{l}\cos(\frac{(2i-1)(2j+2k-2)p\pi}{q})\right]\\ &=\frac{1}{2}\left[\frac{\sin((q-1)\frac{(2j-2k)p\pi}{q})}{2\sin(\frac{(2j-2k)p\pi}{q})}-\frac{\sin((q-1)\frac{(2j-2k)p\pi}{q})}{2\sin(\frac{(2j+2k-2)p\pi}{q})}\right]\\
			&=\frac{1}{2}(-\frac{1}{2}-(-\frac{1}{2}))\\
			&=0,
		\end{aligned}
	\end{equation*}
	where the third equality holds because $\sin\theta \left ( \sum_{i=1}^{l}\cos(2i-1)\theta   \right ) =\frac{1}{2} \sin2l\theta $. For $j=k$, 
	\begin{equation*}\label{eq13}
		\begin{aligned}
			N_{j,j}^{(2)}&=\sum_{i=1}^{\frac{q-1}{2}} \sin^{2}(\frac{(2i-1)(2j-1)p\pi }{q})\\
			&=\frac{1}{2}\sum_{i=1}^{\frac{q-1}{2}}\left(1-\cos((2i-1)\frac{2(2j-1)p\pi}{q})\right)\\
			&=\frac{q-1}{4}-\frac{1}{2}\frac{\sin((q-1)\frac{2(2j-1)p\pi}{q})}{2\sin(\frac{2(2j-1)p\pi}{q})}\\
			&=\frac{q}{4}.
		\end{aligned}
	\end{equation*}
	Thus, $M^{(2)}$ is non-degenerate, and $\mathcal{C}_2$ is a modular tensor category. 
	
	For $\mathcal{C}_3 = \{(1,1), (2,1), (3,1),\ldots, (p-1,1)\}$, the $(i,j)\text{-}$entry of the matrix $S^{(3)}$ is also given by the formula $(4.5)$ in \cite{DJX} as follows:
	$$
	S^{(3)}_{i,j} = \sqrt{\frac{8}{pq}}(-1)^{i+j+1}\sin(\frac{ijq\pi }{p})\sin(\frac{p\pi }{q}).
	$$
	Let $M^{(3)}_{i,j}=(-1)^{i+j+1}\sin(\frac{ijq\pi }{p})$. It suffices for us to prove the non-degeneracy of $M^{(3)}$. Similarly, we consider matrix $N^{(3)}:=(M^{(3)})^{\mathbf{T}}M^{(3)}$. For $j\neq k$, we have
	\begin{equation*}\label{eq14}
		\begin{aligned}
			N^{(3)}_{j,k}&=(-1)^{j+k}\sum_{i=1}^{p-1}\sin(\frac{ijq\pi}{p})\sin(\frac{ikq\pi}{p})\\
			&=(-1)^{j+k}\frac{1}{2}\left[\sum_{i=1}^{p-1}\cos(\frac{i(j-k)q\pi}{p})-\sum_{i=1}^{p-1}\cos(\frac{i(j+k)q\pi}{p})\right]\\
			&=(-1)^{j+k}\frac{1}{4}\left[\frac{\sin((p-\frac{1}{2})\frac{(j-k)q\pi}{p})}{\sin(\frac{(j-k)q\pi}{2p})}-\frac{\sin((p-\frac{1}{2})\frac{(j+k)q\pi}{p})}{\sin(\frac{(j+k)q\pi}{2p})}\right]\\
			&=0,
		\end{aligned}
	\end{equation*}
	where the third equality holds because $\sin\frac{\theta}{2} \left ( \sum_{i=1}^{2l-1}\cos i\theta   \right ) =\frac{1}{2} (\sin(2l-\frac{1}{2})\theta-\sin\frac{\theta}{2}) $. For $ j=k $,
	\begin{equation*}\label{eq15}
		\begin{aligned}
			N_{j,j}^{(3)}&=\sum_{i=1}^{p-1} \sin^{2}(\frac{ijq\pi }{p})\\
			&=\frac{1}{2}\sum_{i=1}^{p-1}\left(1- \cos(\frac{2ijq\pi }{p})\right) \\
			&=\frac{p-1}{2}-\frac{1}{2}\left(\frac{\sin(p-\frac{1}{2})\frac{2jq\pi}{p}}{2\sin\frac{jq\pi}{p}}-\frac{1}{2}\right)\\
			&=\frac{p}{2}.
		\end{aligned}
	\end{equation*}
	Thus, $M^{(3)}$ is non-degenerate, and $\mathcal{C}_3$ is a modular tensor category.
	
	For $\mathcal{C}_1 = \{(1,1), (1,2),(1,3),\ldots, (1,q-1)\} $, the $(i,j)\text{-}$entry of its $S$-matrix is:
	$$
	S^{(1)}_{i,j} = \sqrt{\frac{8}{pq}}(-1)^{i+j+1}\sin(\frac{q\pi }{p})\sin(\frac{ijp\pi }{q}).
	$$
	When $i=q-1$,
	\begin{equation*}
		\begin{aligned}
			S^{(1)}_{q-1,j}&=\sqrt{\frac{8}{pq}}(-1)^{j+1}\sin(\frac{q\pi}{p})\sin(\frac{(q-1)jp\pi}{q})\\
			&=\sqrt{\frac{8}{2l(2l+1)}}(-1)^{j}\sin(\frac{q\pi}{p})\sin(\frac{jp\pi}{q})\\
			&=S^{(1)}_{1,j}.
		\end{aligned}
	\end{equation*}
	This implies that the first row of the matrix $S^{(1)}$ is exactly the same as the last row. Therefore, $S^{(1)}$ is degenerate, and $\mathcal{C}_1$ is not a modular tensor category.
	
	For $\mathcal{C}_4 = \{(1,1), (3,1), (5,1),\ldots, (p-1,1)\}$, the $(i,j)\text{-}$entry of its $S$-matrix is:
	$$
	S^{(4)}_{i,j}=-\sqrt{\frac{8}{pq}}\sin(\frac{p\pi}{q})\sin(\frac{(2i-1)(2j-1)q\pi}{p}).
	$$
	When $i=\frac{p}{2}$,
	\begin{equation*}
		\begin{aligned}
			S^{(4)}_{\frac{p}{2},j} &= -\sqrt{\frac{8}{pq}}\sin(\frac{p\pi }{q})\sin(\frac{(p-1)(2j-1)q\pi }{p})=S^{(4)}_{1,j}.
		\end{aligned}
	\end{equation*}
	Then the first row of the matrix $S^{(4)}$ is exactly the same as the last row. Therefore, $S^{(4)}$ is degenerate, and $\mathcal{C}_4$ is not a modular tensor category. For $ \mathcal{C}_5, \mathcal{C}_6  $, see Proposition $\ref{prop2}$.
	
	Using the same method above and Proposition $\ref{prop2}$, we can prove that when $p$ is odd and $q$ is even with $p\geq5$, the matrices $S^{(2)}, S^{(3)},S^{(5)}$ and $S^{(6)}$ are degenerate, while the matrices $S^{(1)}$ and $S^{(4)}$ are non-degenerate. Therefore, when $p$ is odd and $q$ is even with $p\geq5$, only $\mathcal{C}_1$ and $\mathcal{C}_4$ admit modular tensor category structures.
\end{proof}

\begin{theorem}\label{thm4}
	In category $\mathcal{C}_{p,q}$, when $p$ and $q$ are odd with $p\geq5$, all premodular subcategories (as specified in Proposition \ref{prop1}) $\mathcal{C}_i$ $(i=1,...,6)$ are modular tensor categories.
\end{theorem}
\begin{proof}
	The non-degeneracy of matrices $S^{(2)}$ and $S^{(3)}$ is exactly the same as that proved in Theorem $\ref{thm3}$, while the non-degeneracy of matrices $S^{(5)}$ and $S^{(6)}$ can be found in Proposition $\ref{prop2}$. It then suffices to prove that the matrices $S^{(1)}$ and $S^{(4)}$ are also non-degenerate.
	
	For $\mathcal{C}_1 = \{(1,1), (1,2),(1,3),\ldots, (1,q-1)\} $, let $M^{(1)}_{i,j}=(-1)^{i+j+1}\sin(\frac{ijp\pi }{q})$. We consider matrix $N^{(1)}:=(M^{(1)})^{\mathbf{T}}M^{(1)}$. For $j\neq k$, we have
	\begin{equation*}\label{eq14}
		\begin{aligned}
			N^{(1)}_{j,k}&=(-1)^{j+k}\sum_{i=1}^{q-1}\sin(\frac{ijp\pi}{q})\sin(\frac{ikp\pi}{q})\\
			&=(-1)^{j+k}\frac{1}{2}\left[\sum_{i=1}^{q-1}\cos(\frac{i(j-k)p\pi}{q})-\sum_{i=1}^{q-1}\cos(\frac{i(j+k)p\pi}{q})\right]\\
			&=(-1)^{j+k}\frac{1}{4}\left[\frac{\sin((q-\frac{1}{2})\frac{(j-k)p\pi}{q})}{\sin(\frac{(j-k)p\pi}{2q})}-\frac{\sin((q-\frac{1}{2})\frac{(j+k)p\pi}{q})}{\sin(\frac{(j+k)p\pi}{q})}\right]\\
			&=0.
		\end{aligned}
	\end{equation*}
	For $ j=k $,
	\begin{equation*}\label{eq15}
		\begin{aligned}
			N_{j,j}^{(1)}&=\sum_{i=1}^{q-1} \sin^{2}(\frac{ijp\pi }{q})\\
			&=\frac{1}{2}\sum_{i=1}^{q-1}\left(1- \cos(\frac{2ijp\pi }{q})\right) \\
			&=\frac{q-1}{2}-\frac{1}{2}\left(\frac{\sin(q-\frac{1}{2})\frac{2jp\pi}{q}}{2\sin\frac{jp\pi}{q}}-\frac{1}{2}\right)\\
			&=\frac{q}{2}.
		\end{aligned}
	\end{equation*}
	Thus, $S^{(1)}$ is non-degenerate. 
	
	As for the non-degeneracy of $S^{(4)}$, it can be adapted from the proof of the non-degeneracy of $S^{(2)}$; one only needs to exchange $p$ and $q$.
\end{proof}

It is known that $\text{qdim}_{L(c_{p,q},0)}\,(m,n)=\text{dim}\,(m,n)$ in \cite{DLN} and $\text{qdim}_{L(c_{p,q},0)}\,(m,n)=\\ \text{FPdim}\,(m,n)$ in \cite{DJX} for any simple object $(m,n)$. Then we have that
\begin{equation*}\label{eq0.16}
	\text{FPdim}\,(m,n)=\text{dim}\,(m,n).
\end{equation*}
Furthermore, it follows that $\text{FPdim}\,\mathcal{C}_{p,q}=\text{dim}\,\mathcal{C}_{p,q}$.

\begin{proposition}
	When $p$ is odd and $q$ is even, the category $\mathcal{C}_{p,q}$ is isomorphic to Deligne's product $\mathcal{C}_1\boxtimes\mathcal{C}_4$; when $p$ is even and $q$ is odd, it is isomorphic to Deligne's product $\mathcal{C}_2\boxtimes\mathcal{C}_3$; and when $p$ and $q$ are both odd, it is isomorphic to Deligne's product $\mathcal{C}_2\boxtimes\mathcal{C}_4\boxtimes\mathcal{C}_5$.
\end{proposition}
\begin{proof}
	When $p$ is odd and $q$ is even, $\mathcal{C}_1$ and $\mathcal{C}_4$ are modular tensor categories by Theorem $\ref{thm3}$. It is obvious that there are embeddings of $\mathcal{C}_1$ and $\mathcal{C}_4$ into $\mathcal{C}_{p,q}$, respectively. It is easy to check that $\mathcal{C}_4$ is the centralizer of $\mathcal{C}_1$ in $\mathcal{C}_{p,q}$. According to \cite{DMNO}\cite{M}, the category $\mathcal{C}_{p,q}$ is isomorphic to Deligne's product $\mathcal{C}_1\boxtimes\mathcal{C}_4$. And we have
	\begin{equation*}\label{eq0.17}
		\text{FPdim}(\mathcal{C}_{p,q})=\frac{pq}{8\sin^2\frac{\pi}{p}\sin^2\frac{\pi}{q}}=\frac{q}{2\sin^2\frac{\pi}{q}}\cdot\frac{p}{4\sin^2\frac{\pi}{p}}=\text{FPdim}(\mathcal{C}_1)\text{FPdim}(\mathcal{C}_4).
	\end{equation*}
	
	When $p$ is even and $q$ is odd, this is similar to the previous case; we also have that $\mathcal{C}_{p,q}\cong\mathcal{C}_2\boxtimes\mathcal{C}_3$ and
	\begin{equation*}\label{eq0.18}
		\text{FPdim}(\mathcal{C}_{p,q})=\frac{pq}{8\sin^2\frac{\pi}{p}\sin^2\frac{\pi}{q}}=\frac{q}{4\sin^2\frac{\pi}{q}}\cdot\frac{p}{2\sin^2\frac{\pi}{p}}=\text{FPdim}(\mathcal{C}_2)\text{FPdim}(\mathcal{C}_3).
	\end{equation*}
	
	When $p$ and $q$ are odd, according to the previous results, we can conclude that the category $\mathcal{C}_{p,q}\cong\mathcal{C}_1\boxtimes\mathcal{C}_4$ and $\mathcal{C}_{p,q}\cong\mathcal{C}_2\boxtimes\mathcal{C}_3$. And it follows that $\mathcal{C}_5$ is the centralizer of $\mathcal{C}_2$ in $\mathcal{C}_{1}$ and $\mathcal{C}_5$ is the centralizer of $\mathcal{C}_4$ in $\mathcal{C}_{3}$. Thus, we have $\mathcal{C}_{p,q}\cong\mathcal{C}_1\boxtimes\mathcal{C}_4\cong\mathcal{C}_2\boxtimes\mathcal{C}_5\boxtimes\mathcal{C}_4$ and
	\begin{equation*}\label{eq0.17}
		\text{FPdim}(\mathcal{C}_{p,q})=2\cdot\frac{q}{4\sin^2\frac{\pi}{q}}\cdot\frac{p}{4\sin^2\frac{\pi}{p}}=\text{FPdim}(\mathcal{C}_5)\text{FPdim}(\mathcal{C}_2)\text{FPdim}(\mathcal{C}_4).
	\end{equation*}
\end{proof}

\begin{remark}\label{rmk2.7}\em{
	For each rank, there is an example of  modular tensor category arising from the unitary series (some subcategory of $\mathcal{C}_{p,p+1}$). }
\end{remark}

\begin{remark}\label{not possible}\em{
There does not exist any modular subcategory $\mathcal{D}$ of Virasoro type for which $\text{ord}(T_{\mathcal{D}})=2$. Whether it is $\mathcal{C}_1, \mathcal{C}_4$ or $\mathcal{C}_2,\mathcal{C}_3$—even for $\mathcal{C}$—there exists some object in them whose conformal weight does not belong to $\frac{1}{2}+\mathbb{Z}$. The modular tensor categories with Frobenius-Schur exponent 2 were classified in \cite{WW} in the form of Deligne products of the pointed modular tensor categories. And the pointed modular tensor categories are realized in \cite{N,DNR}. }
\end{remark}

We have already identified all modular tensor subcategories of $\mathcal{C}_{p,q}$. Recalling the results in \cite{DLN}, each such modular tensor category $\mathcal{D}$ gives rise to a projective representation $\bar{\rho}$ of the modular group $\text{SL}_{2}(\mathbb{Z})$ with the Frobenius-Schur exponent $\text{FSexp}(\mathcal{D})=\text{ord}(\bar{\rho}(T))$, and $\bar{\rho}$ admits a lifting whose kernel is a congruence subgroup of level $n:=\text{ord}(\rho(T))$. However, since these modular tensor subcategories are not module categories of a vertex operator algebra, we cannot obtain relevant results about the modular invariance of vertex operator algebras.

If $\mathcal{C}=\mathcal{D}_1\boxtimes\mathcal{D}_2$, and for any simple object $Z\in\mathcal{C}$, there exist simple objects $X\in\mathcal{D}_1, Y\in\mathcal{D}_2$, such that $Z=X\boxtimes Y$, then we have $\theta_Z=\theta_{X\boxtimes Y}=\theta_X\cdot\theta_Y$. The Deligne product of $\mathcal{C}_{p,q}$ under our discussion satisfies exactly the above condition, then we have the following result.
\begin{remark}
  \em{The Frobenius-Schur exponent of $\mathcal{C}_{p,q}$ is equal to the least common multiple of the Frobenius-Schur exponents of all subcategories in its Deligne product decomposition.}
\end{remark}

\section{Another approach to non-degeneracy}
In this section, we shall use a more categorical method to prove Theorem $\ref{thm3}$ and $\ref{thm4}$. It is well known that the premodular tensor category $\mathcal{C}$ is modular if and only if its Müger Center $\mathcal{C}'=<\mathbf{1}>$. It follows by definition that if $X\in\mathcal{C}$ belongs to $\mathcal{C}'$, then for any $Y\in\mathcal{C}$,
\begin{equation*}\label{eq1.1}
	\beta_{Y,X}\circ\beta_{X,Y}=id_{X\otimes Y}.
\end{equation*}
Without abuse of notation, we denote by $\theta_X$ the constant associated with the isomorphism involving the simple object $X$. According to the balanced equation
\begin{equation*}\label{eq1.2}
	\theta_{X\otimes Y}=(\theta_X\otimes\theta_Y)\circ\beta_{Y,X}\circ\beta_{X,Y}.
\end{equation*}
So if $X\in\mathcal{C}'$, then for any $Y\in\mathcal{C}$, it follows that
\begin{equation*}\label{eq1.4}
	\theta_{X\otimes Y}=\theta_X\cdot\theta_Y\cdot id_{X\otimes Y}.
\end{equation*}
By the semisimplicity of $\mathcal{C}$, we have
\begin{eqnarray}\label{eq1.5}
	\nonumber 
	\theta_{X\otimes Y} &=& \sum_{Z\in X\otimes Y}\theta_{Z}\cdot i_{Z}p_{Z}, \\
	\nonumber	id_{X\otimes Y} &=& \sum_{Z\in X\otimes Y}i_{Z}p_{Z}.
\end{eqnarray}
Thus, we can derive the following criterions. 
\begin{lemma}\label{lm1.1}
	If $X\in\mathcal{C}$ is contained in the Müger Center of $\mathcal{C}$, then for any $Y\in\mathcal{C}$, the following holds: 
	\begin{equation*}\label{eq1.6}
		\theta_{Z}=\theta_{X}\cdot\theta_{Y}
	\end{equation*}
	for any summand $Z$ of $X\otimes Y$.
\end{lemma}

\begin{corollary}\label{coro1.2}
	If there exist summands $Z_1,Z_2$ of $X\otimes Y$ such that $\theta_{Z_1}\neq\theta_{Z_2}$, then neither $X$ nor $Y$ is contained in the Müger Center $\mathcal{C}'$.
\end{corollary}

The symbol $\otimes$ above in general category theory corresponds to the symbol $\boxtimes$ in this paper. For the simple current $(1,q-1)=(p-1,1)$, we have $(1,q-1)\boxtimes(1,i)=(1,q-i)$ or $(p-1,1)\boxtimes(j,1)=(p-j,1)$. Then we only need to check
\begin{equation*}\label{eq1.7}
	\theta_{(1,q-i)} = \theta_{(1,q-1)}\cdot\theta_{(1,i)}
\end{equation*}
or
\begin{equation}\label{eq1.8}
	\theta_{(p-j,1)} = \theta_{(p-1,1)}\cdot\theta_{(j,1)}
\end{equation}
where $(1,i)$ or $(j,1)$ belongs to the corresponding subcategory. 

The equation $(\ref{eq1.7})$ holds if and only if $\frac{(i-1)(p-2)}{2}$ is an integer. It means that if $i$ only takes odd integers or $p$ is even, the equation $(\ref{eq1.7})$ holds for a fixed $p$. Therefore, when $p$ is odd and $q$ is even, it follows that $(1,q-1)\notin\mathcal{C}_1'$ and $(1,q-1)\in\mathcal{C}_2'$; when $p$ is even and $q$ is odd, it follows that $(1,q-1)\in\mathcal{C}_1'$ but not in $\mathcal{C}_2'$; 
when $p$ and $q$ are odd, it follows that $(1,q-1)$ is neither contained in $\mathcal{C}_1'$ nor in $\mathcal{C}_2'$.

Similarly, for $(\ref{eq1.8})$ to hold, it is required that $j$ only takes on odd values or $q$ is even. So when $p$ is even and $q$ is odd, it follows that $(p-1,1)\notin\mathcal{C}_3'$ and $(p-1,1)\notin\mathcal{C}_4'$; when $p$ is odd and $q$ is even, it follows that $(p-1,1)\in\mathcal{C}_3'$ but not in $\mathcal{C}_4'$; when $p$ and $q$ are odd, it follows that $(p-1,1)$ is neither contained in $\mathcal{C}_3'$ nor in $\mathcal{C}_4'$.

In subcategory $\mathcal{C}_1$ or $\mathcal{C}_2$, if $(1,i)\in\mathcal{C}_1'$ or $\mathcal{C}_2'$ with $i\neq1,q-1$, then for any $(1,j)\in\mathcal{C}_1$ or $\mathcal{C}_2$,
$$
\theta_{(1,l)}=\theta_{(1,i)}\cdot\theta_{(1,j)},
$$
where $(1,l)$ runs over the summands of $(1,i)\boxtimes(1,j)$. 
When we take $j=i$, $(1,1)$ and $(1,3)$ are both contained in $(1,i)\boxtimes(1,i)$. But $\theta_{(1,1)}=1,\theta_{(1,3)}=
\exp(2\pi i\frac{2p}{q})$. According to Corollary $\ref{coro1.2}$, we conclude that $(1,i)$ is not contained in the Müger Center $\mathcal{C}_1'$ or $\mathcal{C}_2'$. In subcategory $\mathcal{C}_3$ or $\mathcal{C}_4$, we similarly observe that there is no $(j,1)$ belonging to the Müger Center with $j\neq1,p-1$.

In subcategory $\mathcal{C}_5$, we only need to check
$$
\theta_{(1,q-1)\boxtimes(1,q-1)}=\theta_{(1,q-1)}^2,
$$
that is, to check whether $2h_{1,q-1}$ is an integer. Since $h_{1,q-1}=\frac{(p-2)(q-2)}{4}$, it follows that the Müger Center $\mathcal{C}_5'$ is trivial if and only if both $p$ and $q$ are odd, and nontrivial otherwise.

In subcategory $\mathcal{C}_6$, for any $p,q$ and any $(i,j)\in\mathcal{C}_6$ where $(i,j)$ is not isomorphic to the simple current, it also holds that $(1,3)$ and $(3,1)$ are both contained in $(i,j)\boxtimes(i,j)$, and $\theta_{(1,3)}=\exp(2\pi i\frac{2p}{q})\neq\exp(2\pi i\frac{2q}{p})=\theta_{(3,1)}$ since $p$ and $q$ are coprime. Thus, these objects are not contained in the Müger Center $\mathcal{C}_6'$. The simple current is not contained in $\mathcal{C}_6$ if and only if $p$ and $q$ are odd; otherwise, it belongs to $\mathcal{C}_6$. Moreover, if the simple current is contained in $\mathcal{C}_6$, then according to the equation
$$
\theta_{(p-i,j)}=\theta_{(p-1,1)}\cdot\theta_{(i,j)},
$$ 
it must belong to the Müger Center $\mathcal{C}_6'$.

According to the above discussion, we can draw the conclusion that: when $p$ is odd and $q$ is even, only the Müger Centers $\mathcal{C}_1'$ and $\mathcal{C}_4'$ are trivial; when $p$ is even and $q$ is odd, only the Müger Centers $\mathcal{C}_2'$ and $\mathcal{C}_3'$ are trivial; and when $p$ and $q$ are odd, the Müger Centers $\mathcal{C}_i'$ are all trivial for $i\in\{1,2,3,4,5,6\}$. Thus, we get another proof of Theorems $\ref{thm3}$ and $\ref{thm4}$.

\section{Modular tensor subcategories of simple current extension $\widetilde{\mathcal{C}}_p$}
Let $q=p+1$, $L(c_{p,p+1},0)$ is a unitary Virasoro vertex operator algebra for any $p\geq2$. In this section, we will consider the module category of its simple current extension. The properties of its simple current extension has been extensively studied in \cite{LLY}: when $p\equiv1,2\pmod4$, the $V\text{-}$module $\widetilde{V}_p=(1,1)\oplus(1,p)$ admits the structure of a vertex operator algebra, where $(1,1)$ is the vertex operator algebra $V$ and $(1,p)$ is a simple current of $V$. Denote by $\widetilde{\mathcal{C}}_p$ the module category of $\widetilde{V}_p$ which still has a modular tensor category structure according to \cite{HKL}. Our goal is to find the subcategories of $\widetilde{\mathcal{C}}_p$. 

Again by \cite{LLY}, for $p\equiv1\pmod4$, the simple objects of $\widetilde{\mathcal{C}}_p$ are $N_{r,s}=(r,s)\oplus(r,p\!+\!1\!-\!s), (r,\frac{p+1}{2})^+$ and $(r,\frac{p+1}{2})^-$, where $(r,s),(r,p\!+\!1\!-\!s)\in\mathcal{C}_{p,p+1}$ with $1\leq r,s\leq\frac{p-1}{2}$ and $s$ is odd. For $p\equiv2\pmod4$, the simple objects of $\widetilde{\mathcal{C}}_p$ are $N_{r,s}=(r,s)\oplus(p\!-\!r,s), (\frac{p}{2},s)^+$ and $(\frac{p}{2},s)^-$, where $(r,s),(p\!-\!r,s)\in\mathcal{C}_{p,p+1}$ with $1\leq s\leq\frac{p}{2}$, $1\leq r<\frac{p}{2}$ and $r$ is odd. And all fusion rules of simple objects are determined therein.

\begin{proposition}\label{prop2.1}
	There are only two nontrivial premodular subcategories of $\widetilde{\mathcal{C}}_p$ for $p\equiv1,2\pmod4$, $p\geq5$ and $p\neq9,10$.
\end{proposition}
\begin{proof}
	Since the category $\widetilde{\mathcal{C}}_p$ is a modular tensor category, we also use the fusion rules to study the closeness of fusion tensor product. By analogy with Proposition $\ref{prop1}$, we also find horizontal and vertical subcategories.
	
	For $p\equiv1\pmod4$, a horizontal or vertical subcategory only contains $N_{1,s},(1,\frac{p+1}{2})^{\pm}$ or $N_{r,1}$. Denote $\widetilde{\mathcal{D}}$ by possible horizontal subcategory, and set $p=4k+1$. 
	
	If $k$ is odd, $\left((1,\frac{p+1}{2})^{\pm}\right)^*=(1,\frac{p+1}{2})^{\mp}$, and it follows that
	\begin{eqnarray}\label{eq2.1}
		\nonumber 
		(1,\frac{p+1}{2})^+\boxtimes(1,\frac{p+1}{2})^+ &=& N_{1,3}\oplus N_{1,7}\oplus\cdots\oplus N_{1,\frac{p+1}{2}-4}\oplus(1,\frac{p+1}{2})^-, \\
	\nonumber 	(1,\frac{p+1}{2})^+\boxtimes(1,\frac{p+1}{2})^- &=& N_{1,1}\oplus N_{1,5}\oplus\cdots\oplus N_{1,\frac{p+1}{2}-2}, \\
	\nonumber 	(1,\frac{p+1}{2})^-\boxtimes(1,\frac{p+1}{2})^- &=& N_{1,3}\oplus N_{1,7}\oplus\cdots\oplus N_{1,\frac{p+1}{2}-4}\oplus(1,\frac{p+1}{2})^+.
	\end{eqnarray}
	If $N_{1,3}\in\widetilde{\mathcal{D}}$, we can obtain $N_{1,5}\in\widetilde{\mathcal{D}}$ since $N_{1,5}$ is a summand of $N_{1,3}\boxtimes N_{1,3}$. Using the fact that $N_{1,n}\in N_{1,3}\boxtimes N_{1,n-2}$ , it follows that any $N_{1,n}$ belongs to $\widetilde{\mathcal{D}}$ where $5\leq n<\frac{p+1}{2}$ and $n$ is odd. Moreover, $(1,\frac{p+1}{2})^{\pm}\in N_{1,3}\boxtimes N_{1,\frac{p-3}{2}}$. Thus, we obtain
	\begin{equation*}\label{eq2.4}
		\widetilde{\mathcal{C}}_1=\{ N_{1,1},N_{1,3},\dots,N_{1,\frac{p-3}{2}},(1,\frac{p+1}{2})^{\pm}\}.
	\end{equation*}
	If $N_{1,s}\in\widetilde{\mathcal{D}}, N_{1,3}\notin\widetilde{\mathcal{D}}$ with $3<s<\frac{p+1}{2}$ and $s$ is odd, it follows that $N_{1,3}$ is a summand of $N_{1,s}\boxtimes N_{1,s}$ which contradicts the assumption. If only $(1,\frac{p+1}{2})^{+}$ or $(1,\frac{p+1}{2})^{-}$ belongs to $\widetilde{\mathcal{D}}$, we can also obtain $N_{1,3}\in\widetilde{\mathcal{D}}$ by fusion rules above. Thus, there is only one nontrivial horizontal subcategory of $\widetilde{\mathcal{C}}_p$.
	
	If $k$ is even, $\left((1,\frac{p+1}{2})^{\pm}\right)^*=(1,\frac{p+1}{2})^{\pm}$, and it follows that 
	\begin{eqnarray}\label{eq2.5}
	\nonumber 	(1,\frac{p+1}{2})^+\boxtimes(1,\frac{p+1}{2})^+ &=& N_{1,1}\oplus N_{1,5}\oplus\cdots\oplus N_{1,\frac{p+1}{2}-4}\oplus(1,\frac{p+1}{2})^+, \\
	\nonumber 	(1,\frac{p+1}{2})^+\boxtimes(1,\frac{p+1}{2})^- &=& N_{1,3}\oplus N_{1,7}\oplus\cdots\oplus N_{1,\frac{p+1}{2}-2}, \\
	\nonumber 	(1,\frac{p+1}{2})^-\boxtimes(1,\frac{p+1}{2})^- &=& N_{1,1}\oplus N_{1,5}\oplus\cdots\oplus N_{1,\frac{p+1}{2}-4}\oplus(1,\frac{p+1}{2})^-.
	\end{eqnarray}
	Then the situation is similar when $k$ is odd, there is only one horizontal subcategory except $p=9$. When $p=9$, according to the fusion rules above, we obtain
	\begin{align*}
		(1,5)^{+}\boxtimes(1,5)^{+}= & N_{1,1}\oplus(1,5)^{+}, \\
		(1,5)^{+}\boxtimes(1,5)^{-}= & N_{1,3}, \\
		(1,5)^{-}\boxtimes(1,5)^{-}= & N_{1,1}\oplus(1,5)^{-}.
	\end{align*}
	There are three nontrivial horizontal subcategories of $\widetilde{\mathcal{C}}_9$: $\widetilde{\mathcal{C}}_1,\{N_{1,1},(1,5)^+\}$ and $\{N_{1,1},(1,5)^-\}$. 
	
	The situation of the vertical subcategories is slightly different from that of the horizontal ones. It is obvious that
	\begin{equation*}\label{eq2.8}
		\widetilde{\mathcal{C}}_2=\{N_{1,1},N_{2,1},\dots,N_{\frac{p-1}{2},1}\}
	\end{equation*}
	is a nontrivial subcategory of $\widetilde{\mathcal{C}}_p$.
	
	Denote $\widetilde{\mathcal{D}}$ by the other possible vertical subcategory of $\widetilde{\mathcal{C}}_p$. If $N_{2,1}\notin\widetilde{\mathcal{D}}$ and $N_{3,1}\in\widetilde{\mathcal{D}}$, we obtain $N_{5,1}\in\widetilde{\mathcal{D}}$ since $N_{5,1}$ is a summand of $N_{3,1}\boxtimes N_{3,1}$. Repeating this, it seems that we can obtain a new vertical subcategory similar to that in Proposition $\ref{prop1}$. If $N_{m,1}\in\widetilde{\mathcal{D}}$, $m$ is odd and $m\geq\frac{p+3}{4}$, it follows that
	\begin{equation*}
		(m,1)\boxtimes(m,1)=(1,1)\oplus(3,1)\oplus\cdots\oplus(2m-1,1),
	\end{equation*}
	according to the fusion rules of $\mathcal{C}_{p,p+1}$. But since $2m-1>\frac{p-1}{2}$ contradicts the assumption of objects in $\widetilde{\mathcal{C}}_p$, we should add the transformation $\sigma$ to $(2m\!-\!1,1)$:
	$$
	\sigma(2m\!-\!1,1)=(p\!+\!1\!-\!2m,p).
	$$
	And $N_{p+1-2m,1}=(p\!+\!1\!-\!2m,1)\oplus(p\!+\!1\!-\!2m,p)$. From the fusion rules of $\widetilde{\mathcal{C}}_p$, we obtain that $N_{p+1-2m,1}\in N_{m,1}\boxtimes N_{m,1}$, where $p+1-2m$ is even. Then we can conclude that all even terms $N_{even,1}$ belong to $\widetilde{\mathcal{D}}$. Thus, there is only one nontrivial vertical subcategory and there are only two nontrivial subcategories of $\widetilde{\mathcal{C}}_p$ except $p=9$.
	
	For $p\equiv2\pmod 4$, set $p=4k+2$. Whether $k$ is odd or even, there is only one nontrivial horizontal or vertical subcategory except when $p=10$:
	\begin{eqnarray*}
		\widetilde{\mathcal{C}}_1 &=& \{ N_{1,1},N_{1,2},\dots,N_{1,\frac{p}{2}}\}, \\
		\widetilde{\mathcal{C}}_2 &=& \{ N_{1,1},N_{3,1},\dots,N_{\frac{p-4}{2},1},(\frac{p}{2},1)^{\pm}\}.
	\end{eqnarray*}
	Thus, there are only two nontrivial subcategories of $\widetilde{\mathcal{C}}_p$ except $p=10$.
\end{proof}
\begin{remark}
	\em{There are six nontrivial subcategories of $\widetilde{\mathcal{C}}_p$ when $p=9,10$.}
	
	When $p=9$, set $\widetilde{\mathcal{C}}_{3}=\{N_{1,1},(1,5)^+\}$ and $\widetilde{\mathcal{C}}_{4}=\{N_{1,1},(1,5)^-\}$. Then we use $\widetilde{\mathcal{C}}_{3},\widetilde{\mathcal{C}}_{4}$ and $\widetilde{\mathcal{C}}_{2}$ to generate the other subcategories. Thus, $\widetilde{\mathcal{C}}_{5}=(\widetilde{\mathcal{C}}_{2},\widetilde{\mathcal{C}}_{3})$ and $\widetilde{\mathcal{C}}_{6}=(\widetilde{\mathcal{C}}_{2},\widetilde{\mathcal{C}}_{4})$.
	
	When $p=10$, set $\widetilde{\mathcal{C}}_{3}=\{N_{1,1},(5,1)^+\}$ and $\widetilde{\mathcal{C}}_{4}=\{N_{1,1},(5,1)^-\}$. Then we use $\widetilde{\mathcal{C}}_{3},\widetilde{\mathcal{C}}_{4}$ and $\widetilde{\mathcal{C}}_{1}$ to generate the other subcategories. Thus, $\widetilde{\mathcal{C}}_{5}=(\widetilde{\mathcal{C}}_{1},\widetilde{\mathcal{C}}_{3})$ and $\widetilde{\mathcal{C}}_{6}=(\widetilde{\mathcal{C}}_{1},\widetilde{\mathcal{C}}_{4})$.
\end{remark}

\begin{theorem}
	For $p\equiv1,2\pmod 4$, $\widetilde{\mathcal{C}}_1$ and $\widetilde{\mathcal{C}}_2$ are modular tensor categories. Moreover, $\widetilde{\mathcal{C}}_i$ with $i=3,4,5,6$ are modular tensor categories for $p=9,10$.
\end{theorem}
\begin{proof}
	It is obvious that $\widetilde{\mathcal{C}}_1$ and $\widetilde{\mathcal{C}}_2$ are premodular. In order to prove that $\widetilde{\mathcal{C}}_1$ and $\widetilde{\mathcal{C}}_2$ are modular tensor categories, it suffices to prove that $\widetilde{\mathcal{C}}_1$ and $\widetilde{\mathcal{C}}_2$ are non-degenerate. 
	
	For $p\equiv2\pmod4$, the horizontal subcategory is $\widetilde{\mathcal{C}}_1=\{N_{1,1}, \dots, N_{1,\frac{p}{2}}\}$. For any $N_{1,j}\in\widetilde{\mathcal{C}}_{1}$, $N_{1,1}$ and $N_{1,3}$ are summands of $N_{1,j}\boxtimes N_{1,j}$. Since $\theta_{N_{1,1}}\neq\theta_{N_{1,3}}$, it follows from Corollary $\ref{coro1.2}$ that $N_{1,j}$ is not contained in the Müger Center of $\widetilde{\mathcal{C}}_{1}$. By the arbitrariness of $N_{1,j}$, we deduce that $\widetilde{\mathcal{C}}_{1}$ is non-degenerate.
	
	The situation in $\widetilde{\mathcal{C}}_2$ is similar to that in $\widetilde{\mathcal{C}}_{1}$, except for the objects $(\frac{p}{2},1)^+$ and $(\frac{p}{2},1)^-$. Let $p=4k+2$. If $k$ is odd, then $N_{1,1}$ and $N_{5,1}$ belong to $(\frac{p}{2},1)^+\boxtimes(\frac{p}{2},1)^-$. For $p>6$, since $\theta_{N_{1,1}}=1\neq \exp(2\pi i\frac{6}{p})=\theta_{N_{5,1}}$, neither $(\frac{p}{2},1)^+$ nor $(\frac{p}{2},1)^-$ is contained in the Müger Center of $\widetilde{\mathcal{C}}_2$. When $p=6$, $\widetilde{\mathcal{C}}_2=\{N_{1,1},(3,1)^+,(3,1)^-\}$. It follows that $(3,1)^+\boxtimes(3,1)^-=N_{1,1}$, and $\theta_{(3,1)^+}\theta_{(3,1)^-}=\exp(\frac{4\pi i}{3})\neq 1=\theta_{N_{1,1}}$. Thus, $\widetilde{\mathcal{C}}_2$ is also non-degenerate. If $k$ is even, $N_{1,1}$ and $N_{5,1}$ belong to $(\frac{p}{2},1)^{\pm}\boxtimes(\frac{p}{2},1)^{\pm}$. This is entirely consistent with the situation above; therefore, whether $k$ is odd or even, $\widetilde{\mathcal{C}}_2$ is non-degenerate. 
	
	The situation for $p\equiv1\pmod4$ is similar to that for $p\equiv2\pmod4$. Thus, $\widetilde{\mathcal{C}}_1$ and $\widetilde{\mathcal{C}}_2$ are non-degenerate when $p\equiv1,2\pmod 4$. As for $p=9,10$, we can also use Corollary $\ref{coro1.2}$ to show that $\widetilde{\mathcal{C}}_i$ with $i=3,4,5,6$ are non-degenerate.
\end{proof}

\begin{remark}
	\em{When $p\equiv1,2\pmod 4$, $\widetilde{\mathcal{C}}_1$ and $\widetilde{\mathcal{C}}_2$ are modular tensor categories. There are embeddings of $\widetilde{\mathcal{C}}_1$ and $\widetilde{\mathcal{C}}_2$ into $\widetilde{\mathcal{C}}_p$, respectively. And it also follows that $\widetilde{\mathcal{C}}_2$ is the centralizer of $\widetilde{\mathcal{C}}_1$ in $\widetilde{\mathcal{C}}_p$. According to \cite{DMNO, M}, the category $\widetilde{\mathcal{C}}_p$ is isomorphic to Deligne's product $\widetilde{\mathcal{C}}_1\boxtimes\widetilde{\mathcal{C}}_2$.}
\end{remark}

\section{Application}
Recall the Mirror extension mentioned in \cite{L} and \cite{CKM}. Let $\mathcal{D}_1$ and $\mathcal{D}_2$ be module (sub)categories of two rational, $C_2$-cofinite and self-dual vertex operator algebras, and let the functor $\mathcal{F}: \mathcal{D}_1\rightarrow\mathcal{D}_2$ be a braid-reversed equivalence. There exists an extension $A$ of the vertex operator algebra $1_{\mathcal{D}_1}\boxtimes 1_{\mathcal{D}_2}$ in the Deligne product $\mathcal{C}=\mathcal{D}_1\boxtimes\mathcal{D}_2$, and the decomposition of $A$ takes the form
$$
A=\sum_{X\in Irr(\mathcal{D}_1)}X\boxtimes \mathcal{F}(X)',
$$
where $X$ runs over all simple objects of the category $\mathcal{D}_1$.

In the preceding text, we have already identified all subcategories of $\mathcal{C}_{p,p+1}$ and can present modular subcategories of arbitrary rank (see Remark \ref{rmk2.7}). Since all simple objects in such subcategories are self-dual, it is natural to ask: if two modular subcategories have the same rank and fusion rules, can we establish a braid-reversed equivalence between these two subcategories via the correspondence between their simple objects, and further realize the aforementioned such mirror extension in the Deligne product of these two subcategories?

Meanwhile, after we fix two subcategories satisfying the given conditions, for the constructed object $A$ to be an extension of the vertex operator algebra, it must also satisfy the following necessary condition: $\theta_A=\text{id}_A$, that is, for any simple object $X\in\mathcal{D}_1$, the equation
\begin{equation*}\label{eq5.1}
  \theta_{X\boxtimes\mathcal{F}(X)'}=\theta_X\cdot\theta_{\mathcal{F}(X)}=\text{id}_{X\boxtimes\mathcal{F}(X)'}
\end{equation*}
holds. By the definition of the twist, the equation is equivalent to $h_X+h_{\mathcal{F}(X)}\in\mathbb{Z}$ for any simple object $X\in\mathcal{D}_1$. 

Thus, after ``gluing" the algebra, we can first screen it by means of the given conditions. On the other hand, according to \cite{GKO}, in the decomposition of the vertex operator algebra extension in the GKO construction, all simple summands involved are modules of the Virasoro VOA, and we can use the commutant of GKO construction to verify whether the algebra obtained in this manner is indeed a vertex operator algebra.

We now present several concrete examples.
\begin{example}
  The modular subcategories of $\mathcal{C}_{4,5}$ are
  \begin{eqnarray*}
    \mathcal{C}_2 &=& \{(1,1),(1,3)\}, \\
    \mathcal{C}_3 &=& \{(1,1),(2,1),(3,1)\}.
  \end{eqnarray*}
  The modular subcategories of $\mathcal{C}_{5,6}$ are
  \begin{eqnarray*}
    \mathcal{C}_1 &=& \{(1,1),(1,2),(1,3),(1,4),(1,5)\}, \\
    \mathcal{C}_4 &=& \{(1,1),(3,1)\}.
  \end{eqnarray*}
  The modular subcategories of $\mathcal{C}_{6,7}$ are
  \begin{eqnarray*}
    \mathcal{C}_2 &=& \{(1,1),(1,3),(1,5)\}, \\
    \mathcal{C}_3 &=& \{(1,1),(2,1),(3,1),(4,1),(5,1)\}.
  \end{eqnarray*}
  The modular subcategories of $\mathcal{C}_{7,8}$ are
  \begin{eqnarray*}
    \mathcal{C}_1 &=& \{(1,1),(1,2),(1,3),(1,4),(1,5),(1,6),(1,7)\}, \\
    \mathcal{C}_4 &=& \{(1,1),(3,1),(5,1)\}.
  \end{eqnarray*}
  
  $(1)$ The subcategories $\mathcal{C}_2$ of $\mathcal{C}_{4,5}$ and $\mathcal{C}_4$ of $\mathcal{C}_{5,6}$ have the same rank and fusion rules, and thus we have
  $$
  A=h^{4,5}_{1,1}\boxtimes h^{5,6}_{1,1}\oplus h^{4,5}_{1,3}\boxtimes h^{5,6}_{3,1},
  $$
  where $h^{4,5}_{1,1}+h^{5,6}_{1,1}=0, h^{4,5}_{1,3}+h^{5,6}_{3,1}=2$. $A$ is a vertex operator algebra appearing in the GKO construction of $\mathcal{L}_{\widehat{sl_2}}(1,0)^{\otimes4}$.
  
  $(2)$ The subcategories $\mathcal{C}_1$ of $\mathcal{C}_{5,6}$ and $\mathcal{C}_3$ of $\mathcal{C}_{6,7}$ have the same rank and fusion rules, and thus we have
  $$
  A=h^{5,6}_{1,1}\boxtimes h^{6,7}_{1,1}\oplus h^{5,6}_{1,2}\boxtimes h^{6,7}_{2,1}\oplus h^{5,6}_{1,3}\boxtimes h^{6,7}_{3,1}\oplus h^{5,6}_{1,4}\boxtimes h^{6,7}_{4,1}\oplus h^{5,6}_{1,5}\boxtimes h^{6,7}_{5,1},
  $$
  where $h^{5,6}_{1,1}+h^{6,7}_{1,1}=0, h^{5,6}_{1,2}+h^{6,7}_{2,1}=\frac{1}{2}, h^{5,6}_{1,3}+h^{6,7}_{3,1}=2, h^{5,6}_{1,4}+h^{6,7}_{4,1}=\frac{9}{2}, h^{5,6}_{1,5}+h^{6,7}_{5,1}=8$. Thus, as some of these sums are not integers, $A$ is not a vertex operator algebra.
  
  $(3)$ The subcategories $\mathcal{C}_2$ of $\mathcal{C}_{6,7}$ and $\mathcal{C}_4$ of $\mathcal{C}_{7,8}$ have the same rank and fusion rules, and thus we have
  $$
  A=h^{6,7}_{1,1}\boxtimes h^{7,8}_{1,1}\oplus h^{6,7}_{1,3}\boxtimes h^{7,8}_{3,1}\oplus h^{6,7}_{1,5}\boxtimes h^{7,8}_{5,1},
  $$
  where $h^{6,7}_{1,1}+h^{7,8}_{1,1}=0, h^{6,7}_{1,3}+h^{7,8}_{3,1}=2, h^{6,7}_{1,5}+h^{7,8}_{5,1}=8$. $A$ is a vertex operator algebra appearing in the GKO construction of $\mathcal{L}_{\widehat{sl_2}}(1,0)^{\otimes6}$.
\end{example}

By \cite{GKO} and \cite{JL}, we have 
$$
\mathcal{L}_{\widehat{sl_2}}(1,0)^{\otimes(n+1)}=\sum_{0\leq2k\leq n+1}\mathcal{L}_{\widehat{sl_2}}(n+1,2k)\otimes M^{(n)}(2k),
$$
where $M^{(n)}(2k)=\oplus_{0\leq 2l\leq n}M^{(n-1)}(2l)\otimes L(c_{n+2,n+3},h_{2l+1,2k+1})$, $M^{(0)}(0)=\mathbb{C}$, and $k$ is a non-negative integer. Moreover, the space $M^{(n-1)}(0)\otimes L(c_{n+2,n+3},0)$ is a full vertex operator subalgebra of $M^{(n)}(0)$. Then we have
\begin{equation*}\label{eq5.2}
  \begin{split}
     M^{(n)}(0) & = \text{Com}_{\mathcal{L}_{\widehat{sl_2}}(1,0)^{\otimes(n+1)}}(\mathcal{L}_{\widehat{sl_2}}(n+1,0))\\
       & =\bigoplus_{0\leq 2l\leq n}\left[\bigoplus_{0\leq 2t\leq n-1}M^{(n-2)}(2t)\otimes L(c_{n+1,n+2},h_{2t+1,2l+1})\right]\otimes L(c_{n+2,n+3},h_{2t+1,1}) \\
       & =\bigoplus_{0\leq 2t\leq n-1}M^{(n-2)}(2t)\otimes\left[\bigoplus_{0\leq 2l\leq n}L(c_{n+1,n+2},h_{2t+1,2l+1})\otimes L(c_{n+2,n+3},h_{2t+1,1})\right],
  \end{split}
\end{equation*}
where $\text{Com}_{V}(U)$ denotes the commutant vertex operator algebra of $U$ in $V$.

It implies that $M^{(n)}(0)$ has a full vertex operator subalgebra $\widetilde{N}^{(n)}:=M^{(n-2)}(0)\otimes(\oplus_{0\leq 2l\leq n}L(c_{n+1,n+2},h_{1,2l+1})\otimes L(c_{n+2,n+3},h_{2l+1,1}))$, and $\widetilde{N}^{(n)}$ also has a vertex operator subalgebra $N^{(n)}:=\oplus_{0\leq 2l\leq n}L(c_{n+1,n+2},h_{1,2l+1})\otimes L(c_{n+2,n+3},h_{2l+1,1})$. According to \cite{L} and \cite{Mc}, we known that $\{L(c_{n+1,n+2},h_{1,2l+1})\vert 0\leq 2l\leq n\}$ and ${L(c_{n+2,n+3},h_{2l+1,1})\vert 0\leq 2l\leq n}$ are subcategories of the corresponding category. 

The first of the above subcategory is the subcategory $\mathcal{C}_2$ of $\mathcal{C}_{n+1,n+2}$ and the second is the subcategory $\mathcal{C}_4$ of $\mathcal{C}_{n+2,n+3}$; specifically, when $p=n+2$ is odd, these subcategories are non-degenerate, and when $p=n+2$ is even, these subcategories are degenerate.

\begin{remark}
\em{  Vertex operator algebras can appear not only by ``gluing" two modular tensor subcategories but also by ``gluing" two premodular subcategories.}
\end{remark}

Thus, we can provide a series of examples of commutant vertex operator algebras based on the results derived from the aforementioned subcategories.

\begin{remark}
\em{Extensions with preliminary progress, including the premodular and modular tensor subcategories within the module categories of the affine vertex operator alge-\hspace{0pt}\\bras $\mathcal{L}_{\widehat{\mathfrak{sl}_n}}(k,0)$ for integer $n\geq 2$ at any level $k$, as well as those in the module categories of the parafermion vertex operator algebra $K(\mathfrak{sl}_n,k)$, will be presented in detail, along with the full proofs of the results, in a forthcoming paper.}
\end{remark}

\section*{Acknowledgement}
Y. Xiao is supported by the National Natural Science Foundation of
China (No. 12401034) and the Natural Science Foundation of Shandong Province (No.
ZR2023QA066). W. Zheng is supported by the National Natural Science Foundation of
China (No. 12201334) and the Natural Science Foundation of Shandong Province (No. ZR2022QA023).

\end{document}